\documentclass[10pt, oneside]{article}    
\usepackage[english]{babel}
\usepackage[toc,page]{appendix}
\usepackage[margin=1in]{geometry}  
\usepackage{hyperref}
\usepackage{graphicx,wrapfig,lipsum}	
\usepackage{enumitem}
\usepackage[labelformat=empty]{caption}
\usepackage{yfonts}	
\usepackage{fixmath}
\usepackage{amssymb}
\usepackage{amsmath}
\usepackage{amsthm}
\usepackage{scalerel}
\usepackage{verbatim}

\newcommand{\Irrep}{\text{Irrep}}

\newcommand{\llambda}{\mathbold{\lambda}}
\newcommand{\mmu}{\mathbold{\mu}}

\usepackage{framed}

\usepackage[colorinlistoftodos]{todonotes}

\usepackage{ytableau}
\usepackage{thm-restate}

\usepackage{mathtools}

\definecolor{lightgreen}{rgb}{0,0.8,0}
\definecolor{darkgreen}{rgb}{0,0.3,0}
\definecolor{lightblue}{rgb}{0,0,0.65}
\definecolor{darkblue}{rgb}{0,0,0.4}
\definecolor{lightred}{rgb}{0.8,0,0}
\definecolor{darkred}{rgb}{0.3,0,0}

\newtheorem{thm}{Theorem}[section]
\newtheorem{lem}[thm]{Lemma}
\newtheorem{prop}[thm]{Proposition}
\newtheorem{cor}[thm]{Corollary}

\theoremstyle{definition}
\newtheorem{defn}{Definition}

\newcommand{\Ind}{\text{Ind}}

\ytableausetup{nosmalltableaux}
\ytableausetup{boxsize=1em}

\title{Tensor Quasi-Random Groups}
\author{Mark Sellke}
\date{}							
\begin{document}
\maketitle

\abstract{\noindent In \cite{qrgroups}, Gowers elegantly characterized the finite groups $G$ in which $A_1A_2A_3=G$ for any positive density subsets $A_1,A_2,A_3$. This property, \emph{quasi-randomness}, holds if and only if $G$ does not admit a nontrivial irreducible representation of constant dimension. We present a dual characterization of \emph{tensor quasi-random} groups in which multiplication of subsets is replaced by tensor product of representations.}

\section{Introduction}

Many large finite groups $G$ exhibit expansion and mixing phenomena. The former states that product sets $A_1A_2\subseteq G$ are always significantly larger than either of $A_1,A_2\subseteq G$. The latter states that when $A_1,A_2,A_3\subseteq G$ are fairly large, then the number of solutions to $a_1a_2=a_3$ for $a_i\in A_i$ is close to $\frac{|A_1A_2A_3|}{|G|^2}.$ Such properties have attracted a lot of interest and are related to theoretical computer science via the notion of expander graphs \cite{margulis1988explicit,lubotzky1988ramanujan}. Seminal papers have established such properties for certain simple groups, see for instance \cite{helfgott2008growth,bourgain2012expansion,breuillard2013expansion,pyber2016growth,eberhard2016product}. 

The present work is inspired by a striking result of \cite{qrgroups} which qualitatively characterizes the groups $G$ that exhibit expansion on large scales. More precisely, it characterizes the groups for which any three subsets $A_1,A_2,A_3$ of constant density must multiply to cover all of $G$. This is equivalent to stating that if $|A_1|,|A_2|=\Omega(|G|)$ then $|A_1A_2|=(1-o(1))|G|$. Throughout we consider the regime of large groups with $|G|\to\infty$ as all other parameters are fixed. We use the standard notations $A=O(B)$ or $B=\Omega(A)$ to indicate that $\frac{A}{B}$ is bounded, and $A=o(B)$ to indicate that $\frac{A}{B}$ tends to $0$.

\begin{thm}[\cite{qrgroups}]

\label{thm:QR}

Let $G$ be a finite group. The following are asymptotically equivalent up to dependence of constants, and define a quasi-random group:

\begin{enumerate}

\item $G$ has no $O(1)$-dimensional non-trivial irreducible representations.

\item If $A_1,A_2,A_3\subseteq G$ each have size $\Omega(|G|)$, then $A_1\cdot A_2\cdot A_3=G$.

\item Fix any constant $m\geq 3$. If $A\subseteq G$ has $|A|=\Omega(|G|)$, then $A^m=A\cdot A\cdot \dots\cdot A=G$.

\item $G$ has neither an $O(1)$-size nor an abelian non-trivial quotient.

\end{enumerate}

\end{thm}

We explain in the Appendix why the above equivalence follows from \cite{qrgroups} as it is not stated directly. It follows from criterion $4$ that large non-abelian simple groups are quasi-random. This implies for instance that given a subset $S\subseteq G$ for $G$ simple, to show that $S^k=G$ for $k$ not too large, it suffices to show that a small power of $S$ has macroscopic $\Omega(|G|)$ size. Such an argument was used in \cite{babai2008product} to simplify the proof of \cite{helfgott2008growth}. Of course the result above is interesting in its own right; the original application was to that constant-density product-free sets do not exist in general finite groups.

 Our purpose is to study a dual problem: given large $G$-representations $V_i$, when must $V_1\otimes V_2\otimes V_3$ contain all irreducible $G$-representations as subrepresentations? In such a case we say this tensor product \emph{covers} $\Irrep(G)$, the set of (isomorphism classes of) irreducible $G$-representations. Another line of working on covering $\Irrep(G)$ is the Saxl conjecture, which asserts that $\Irrep(S_n)$ can be covered by a tensor square for $n$ large enough - see \cite{Pak,IkenDom,li2018,LuoSellke}. The work \cite{steinbergsquare} establishes such a result in groups of Lie type, and \cite{liebeck,sellke20,liebeck2021} study the number of tensor powers of a fixed irreducible representation needed to cover $\Irrep(G)$ in various cases. As we allow our tensor factors $V_i$ to be reducible, it is not obvious how to best measure their size. We will use the Plancherel measure.

\begin{defn}

For $G$ a finite group, the Plancherel measure $M_G$ is a probability distribution on $\Irrep(G)$ which assigns the irreducible representation $\llambda$ probability $M_G(\llambda)=\frac{\dim(\llambda)^2}{|G|}$. For an arbitrary finite-dimensional $G$-representation $V$, let $M_G(V)$ denote the Plancherel measure of the set of distinct (up to isomorphism) irreducible subrepresentations $\llambda\subseteq V$.

\end{defn}

Our main result, a dual version of Theorem~\ref{thm:QR}, characterizes which groups exhibit good tensor product expansion on large scales.

\begin{thm}

\label{thm:TQR}

Let $G$ be a finite group. The following are asymptotically equivalent up to dependence of constants, and define a tensor quasi-random or TQR group: 

\begin{enumerate}

\item \label{item:nosmallconj} $G$ contains no $O(1)$ sized non-trivial conjugacy class.

\item \label{item:tensorprod} If $V_1,V_2,V_3$ are $G$-representations with $M_G(V_i)=\Omega(1)$ for all $i$, then $V_1\otimes V_2\otimes V_3$ covers $\Irrep(G)$.

\item \label{item:tensorpower} Fix $m\geq 3$. If $V$ is a $G$-representation with $M_G(V)=\Omega(1)$, then $M_G(V^{\otimes m})>\frac{1}{2}$.

\item \label{item:subgroup} $G$ contains neither a $O(1)$ size nontrivial normal subgroup, nor a $O(1)$ index normal subgroup (possibly equal to $G$) with non-trivial center.

\end{enumerate}

\end{thm}

While \cite{qrgroups} turns subsets $A\subseteq G$ into representation-theoretic data via the Fourier transform, we turn representations into class functions via their characters. The crucial lowest dimension of a non-trivial irreducible representation is for us replaced by the smallest non-trivial conjugacy class of $G$. In both situations, a key insight is that if all nontrivial irreducible representations or conjugacy classes are large, then an appropriate $\ell^{\infty}$ norm must be quite small - see the proof of Theorem~\ref{thm:22infty}.

We remark that from the fourth criterion in Theorems~\ref{thm:TQR} it follows that large simple groups are also TQR. In \cite{liebeck} it was conjectured that for simple groups, any irreducible representation $\llambda$ requires only $O\left(\frac{\log|G|}{\log\dim\llambda}\right)$ tensor powers to cover $\Irrep(G)$ - this is an easy lower bound since covering $\Irrep(G)$ requires large dimension. This conjecture was proved there in bounded rank groups of Lie type, and then for $S_n$ in \cite{sellke20} and subsequently $A_n$ in \cite{liebeck2021}. Because simple groups are TQR, to establish the conjecture it suffices to show that a small tensor power of $\llambda$ has Plancherel measure $\Omega(1)$ (or even a bit smaller depending on the group in question), at which point one could apply Theorem~\ref{thm:TQR} to finish. This idea was used in \cite{babai2008product} to simplify the landmark result of \cite{helfgott2008growth} on covering $SL_2(\mathbb Z/p\mathbb Z)$ by products of subsets. As in \cite{babai2008product}, our methods do not seem helpful for showing growth at small scales, but are only able to show that products of large representations quickly cover everything.

In Section~\ref{sec:markov} we specialize our results to the \emph{tensor product Markov chains} studied in \cite{fulman,fulman2,fulman3,markov}. The result is that the tensor quasi-randomness of a group $G$ characterizes whether certain tensor product chains mix in constant time. 

\begin{cor}
\label{cor:mix}
Associate with a $G$-representation $V$ its reduced representation $\widetilde V$ as in Definition~\ref{defn:reduced}. For any $\varepsilon>0$, if $G$ is large and TQR, then when $M_{G}(V)=\Omega(1)$ the uniform $\varepsilon$-mixing time of the tensor product Markov chain given by $\cdot\otimes \widetilde V$ is at most $3$ . Conversely if $G$ is large and not TQR then the total variation $\frac{1}{4}$-mixing times of the chains $\cdot\otimes \widetilde V$ are arbitrarily large for suitable $G$-representations $V$.

\end{cor}

\subsection{Preliminaries}

\begin{defn}

Direct sum and tensor product of class functions are defined by element-wise sum/product (the notation is chosen to emphasize the underlying representations). For any character $\chi$ we denote by $\chi_0$ the function $\chi_0(g)=1_{g\neq e}\cdot \chi(g)$, so that $\chi(g)=\chi(e)\cdot 1_{g=e}+\chi_0(g)$.

\end{defn}

In \cite{qrgroups}, subsets $S\subseteq G$ correspond to their characteristic functions $1_{S}(x)=1_{x\in S}$, and Fourier analysis is performed on these functions. For us, the corresponding ``right" version of a $G$-representation $V$ is the \emph{reduced representation}. Throughout the paper we identify $G$-representations with their set of isomorphism classes of irreducible representations. We will use $\ell^p$ norms $|\cdot|_p$ on functions $f:G\to\mathbb C$ with counting measure on $G$, so that $|f|_{p}=\left(\sum_{g\in G}|f(g)|^p\right)^{1/p}.$ 

\begin{defn}
\label{defn:reduced}
For $V$ a $G$-representation, the corresponding reduced representation is given by

\[\widetilde V=\bigoplus_{\llambda\in V}\dim(\llambda)\cdot\llambda.\]

Hence $\widetilde V$ depends only on the set of distinct irreducibles contained in $V$, and takes $\Irrep(G)$ to the regular representation. We also define the reduced character function $\widetilde{\chi}^V:G\to\mathbb C$ via:

\[\widetilde{\chi}^V(g)=\frac{1}{|G|}\cdot \chi^{\widetilde V}(g)=\frac{1}{|G|}\sum_{\llambda\in\Irrep(G)} 1_{\llambda\in V}\cdot\dim(\llambda)\cdot \chi^{\llambda}(g).\]

\end{defn}

The regular representation has reduced character $1_{g=e}$. In general $\widetilde{\chi}^V(e)=M_G(V)$ and $|\widetilde{\chi}^V(e)|_{2}=\sqrt{M_G(V)}$. The following covering criterion shows that if the $\ell^1$ mass of a character $\chi$ is concentrated on $\chi(e)$, then the corresponding representation must contain all irreducibles.

\begin{lem}

Let $V$ be a $G$-representation and let $\chi:G\to \mathbb C$ be a class function in the $\mathbb C$-linear span of $\{\chi^{\llambda}|\llambda\in V\}$. Suppose also that $|\chi(e)|>\sum_{g\in G\backslash\{e\}}|\chi(g)|.$ Then $S=\Irrep(G)$.

\end{lem}

\begin{proof}

We show $\langle \chi,\chi^{\llambda}\rangle\neq 0$ for any $\llambda\in\Irrep(G)$. Indeed,  

\[\langle \chi,\chi^{\llambda}\rangle = \chi(e)\chi^{\llambda}(e)+\langle \chi_0,\chi^{\llambda}_0\rangle.\]

Now by assumption $|\chi(e)|>|\chi_0|_{\ell^1}$, and since $\llambda$ is a genuine $G$-representation $|\chi^{\llambda}(e)|\geq |\chi^{\llambda}_0|_{\ell^\infty}.$ Therefore 

\[|\langle \chi_0,\chi^{\llambda}_0\rangle|\leq |\chi_0|_1\cdot |\chi^{\llambda}_0|_{\infty}<|\chi(e)\chi^{\llambda}(e)|.\]

Hence $\langle \chi,\chi^{\llambda}\rangle\neq 0$ completing the proof.

\end{proof}

\section{Covering $Irrep(G)$ when all Conjugacy Classes are Large}

In this section we prove the implication $\ref{item:nosmallconj}\implies \ref{item:tensorprod}$ of Theorem~\ref{thm:TQR}. 

As a warmup we reprove a result from our previous work \cite{sellke20} on covering $\Irrep(G)$ by a tensor product of two representations. We consider this a dual to the trivial statement that if $A_1,A_2\subseteq G$ satisfy $|A_1|+|A_2|>|G|$ then $A_1A_2=G$. 

\begin{thm}[\cite{sellke20}]

Suppose $M_G(V_1)+M_G(V_2)>1$. Then $V_1\otimes V_2$ covers $\Irrep(G)$. 

\end{thm}

\begin{proof}

For convenience take $a_i=M_G(V_i)$. Then

\[\left(\widetilde{\chi}^{V_1}\otimes \widetilde{\chi}^{V_2}\right)(g)=a_1a_2\cdot 1_{g=e}+\widetilde{\chi}^{V_1}_0\otimes \widetilde{\chi}^{V_2}_0(g).\]

The Cauchy-Schwarz inequality implies:

\[|\widetilde{\chi}^{V_1}_0\otimes \widetilde{\chi}^{V_2}_0|_1\leq |\widetilde{\chi}^{V_1}_0|_2 |\widetilde{\chi}^{V_2}_0|_2= \sqrt{(a_1-a_1^2)(a_2-a_2^2)}.\]

The assumption $a_1+a_2>1$ implies $a_1>1-a_2$ and $a_2>1-a_1$. Therefore $a_1a_2>(1-a_1)(1-a_2)$, and so $a_1^2a_2^2>(a_1-a_1^2)(a_2-a_2^2).$ Therefore $\left(\widetilde{\chi}^{V_1}\otimes \widetilde{\chi}^{V_2}\right)(e)>0$, concluding the proof.

\end{proof}












Let $c(G)$ denote the minimal size of a non-trivial conjugacy class in $G$. We now show that tensor triple products cover $\Irrep(G)$ when $c(G)$ is large. The idea is to use Holder's inequality with $\frac{1}{2}+\frac{1}{2}+\frac{1}{\infty}=1$.

\begin{lem}\label{lem:cG}

Let $V$ be a $G$-representation, then $|\widetilde{\chi}^V_0|_{\infty}\leq c(G)^{-1/2}.$

\end{lem}

\begin{proof}

Certainly $|\widetilde{\chi}^V_0|_2\leq 1$. Moreover $\widetilde{\chi}^V_0$ is constant on conjugacy classes, all of which have size at least $c(G)$. This implies the claim.

\end{proof}

\begin{thm}

\label{thm:22infty}

For representations $V_1,V_2,V_3$ of $G$, if \[M_G(V_1)M_G(V_2)M_G(V_3)> c(G)^{-1/2}\] then $V_1\otimes V_2\otimes V_3$ contains all irreducible representations of $G$. 

\end{thm}

\begin{proof}

We consider the character values on $e\in G$ and on all other values. $\widetilde{\chi}^{V_i}(e)=M_G(V_i)$ and so \[\widetilde{\chi}^{V_1}\otimes \widetilde{\chi}^{V_2}\otimes \widetilde{\chi}^{V_3}(e)=M_G(V_1)M_G(V_2)M_G(V_3).\] On the other hand, Holder's inequality and Lemma~\ref{lem:cG} imply:

\[|\widetilde{\chi}^{V_1}_0\otimes \widetilde{\chi}^{V_2}_0\otimes \widetilde{\chi}^{V_3}_0|_1 \leq |\widetilde{\chi}^{V_1}_0|_2\cdot |\widetilde{\chi}^{V_2}_0|_2\cdot |\widetilde{\chi}^{V_3}_0|_{\infty}
\leq c(G)^{-1/2}.\]

\end{proof}

This establishes the implication $\ref{item:nosmallconj}\implies \ref{item:tensorprod}$ of Theorem~\ref{thm:TQR}. In fact examination of the proof above shows slightly more, in analogy with Theorem 3.3 of \cite{qrgroups}. Let $V_1,V_2,V_3$ be representations of a TQR group $G$ with Plancherel measures $M_G(V_i)=a_i$. Then the multiplicity of an irreducible $\llambda$ in $\widetilde V_1\otimes \widetilde V_2\otimes \widetilde V_3$ is $|G|^2(a_1a_2a_3 \dim(\llambda)+|\widetilde{\chi}^{V_1}_0\otimes \widetilde{\chi}^{V_2}_0\otimes \widetilde{\chi}^{V_3}_0|_1)$, where the last term is at most $c(G)^{-1/2}$. Hence in a TQR group, such a tensor product of reduced representations contains every irreducible $\llambda$ a number of times approximately proportional to $\dim(\llambda)$. The $L^{\infty}$ mixing time result is a slight generalization of this argument.

\section{Remaining Proofs for Theorem~\ref{thm:TQR}}

The implication $\ref{item:tensorprod}\implies \ref{item:tensorpower}$ of Theorem~\ref{thm:TQR} is clear. We now show the implication $\ref{item:subgroup}\implies \ref{item:nosmallconj}$ of Theorem~\ref{thm:TQR}, and in the next subsection show $\ref{item:tensorpower} \implies \ref{item:subgroup}$. 

\begin{proof}[Proof of Theorem~\ref{thm:TQR}, implication $\ref{item:subgroup}\implies \ref{item:nosmallconj}$]

We go by contradiction and assume $G$ contains neither a small non-trivial conjugacy class $C$ at most $k=O(1)$ in size, nor a normal subgroup at most $k!$ in size. We will show that $G$ contains a constant-index normal subgroup with non-trivial center.

First, if $k=1$ then $G$ itself has nontrivial center so we may assume $k>1$. Then $G$ acts on $C$ by conjugation; this defines a non-trivial homomorphism $\phi:G\to S_k.$ Therefore $N=\ker\phi\trianglelefteq G$ is a normal subgroup of $G$ commuting with each element of $C$. Its index is $|G/N|\leq k!$. The subgroup $H$ generated by $C$ is normal (as $C$ is a conjugacy class) and commutes with $N$. Since $H$ is a normal subgroup of $G$, by assumption $|H|>k!$. Let $K=H\cap N$. We have 

\[|K|\geq \frac{|H||N|}{|G|}\geq  \frac{|H|}{k!} >1.\]

Since $H$ and $N$ commute, $K$ is central in $N$. Therefore $G$ contains a constant index subgroup $N$ with nontrivial center $K$. This shows $\ref{item:subgroup}\implies \ref{item:nosmallconj}$.

\end{proof}

\subsection{Tensor power condition implies normal subgroup condition}

To prove the implication $\ref{item:tensorpower} \implies \ref{item:subgroup}$ we begin with two preparatory lemmas in additive combinatorics. Throughout, if $n$ is a positive integer and $B$ is a subset of an abelian group, we denote by $nB$ the $n$-fold sum $B+B+\dots+B$.

\begin{lem}

\label{lem:additive combo}

Let $k,m,n$ be positive integers and let $B\subseteq K$ be a subset of an abelian group $K$ with size $|B|=k+1$. Then $mnB$ can be covered by $(10km)^{k}$ translates of $nB$. In particular $\frac{|mnB|}{|nB|}\leq (10km)^{k}.$

\end{lem}

\begin{proof}

It suffices to show the result when $B=\{0,e_1,\dots,e_{k}\}$ consists of zero and standard basis vectors in $K=\mathbb Z^k$. Indeed given any such identification, the resulting homomorphism from $\mathbb Z^{k}\to K$ transfers a covering inside $\mathbb Z^k$ to a covering inside $K$. Observe that $nB\supseteq \{0,\dots,
j-1\}^{k}$ for $j=1+\lfloor \frac{n}{k}\rfloor$, while $mnB\subseteq \{0,\dots,J-1\}^{k}$ for $J=mn+1$. As $J\leq 10kmj$ the result follows.

\end{proof}

\begin{lem}

\label{lem:sumset}

Let $K$ be a finite non-trivial abelian group admitting an action of automorphisms by a group $L$ of size $k$. Then there exists a subset $A\subseteq K$ with $|A|=\Omega_{k,m}(|K|)$ and $|mA|\leq \frac{|K|}{2}$.

\end{lem}

\begin{proof}

Fix a small, positive value $\varepsilon<\frac{1}{(10km)^{k+1}}.$ We construct $A$ using the following iterative algorithm initialized with $A=\{e\}$ and $a\in K\backslash \{e\}$. Throughout, for $a\in K$, we denote by $La\subseteq K$ the orbit of $a$ under the action of $L$.

\begin{enumerate}
    \item If $A\neq A+La$, update $A\leftarrow A\cup (A+La)$.
    \item If $A=A+La$, update $a$ to any element of $K\backslash A$ and return to step $1$.
    \item At the first time that $|A|\geq \varepsilon|K|$, terminate and output the current set $\widehat A$.
\end{enumerate}

This algorithm repeats multiple iterations of step $1$ interrupted by single iterations of step $2$ until terminating in step $3$. To show that the outputed set $\widehat A$ satisfies the conditions of the lemma it remains to show that $mA\subsetneq K$. First observe that the size $|A|$ grows by a factor of at most $k+1$ at each iteration because $|La|\leq |L|\leq k$ for each $a\in K$. 

We claim that all times, $|mA|\leq (10km)^k|A|$ holds. Indeed let $a_1,a_2,\dots,a_t$ be the sequence of values of $a$ so far in the algorithm, and let $A_{t-1}$ be the set $A$ at the time that $a\leftarrow a_t$ was updated. Then $A_{t-1}$ is exactly the subgroup generated by the sets $La_i$ for $i\leq t-1$. Letting $B=La_t\cup e$, the current set $A$ may be written as

\[A=A_{t-1}+nB\]

for some non-negative integer $n$, where $B=e\cup La$ has size $|B|=k+1$. Since $A_{t-1}\subseteq K$ is a subgroup, we may set $\psi:K\to K/A_{t-1}$ to be the natural quotient map. Then Lemma~\ref{lem:additive combo} yields $\frac{|mA|}{|A|}=\frac{|nm\psi(B)|}{|\psi(nB)|}\leq (10km)^k$ proving the claim.

Since $|A|$ grows by a factor of at most $k+1$ each iteration, as long as $|K|\geq \varepsilon^{-1}$ the output set $\widehat A$ satisfies $\frac{|\widehat{A}|}{|K|}< \frac{1}{10\cdot(10km)^k}.$ Therefore $|m\widehat A|< \frac{|K|}{10}$ as desired. On the other hand if $|K|\leq  \varepsilon^{-1}$ then step $3$ immediately outputs $\widehat A=\{e\}$. In this case $|m\widehat A|=1$, which also concludes the proof as $K$ is assumed to be non-trivial.

\end{proof}

We now complete the proof of Theorem~\ref{thm:TQR}.

\begin{proof}[Proof of Theorem~\ref{thm:TQR}, implication $\ref{item:tensorpower}\implies \ref{item:subgroup}$]

It is easy to see that if $G$ has a constant size normal subgroup $N$, then condition \ref{item:tensorpower} fails - simply take $V$ to be the regular representation on the quotient $G/N$, viewed as a $G$-representation. Then in $V^{\otimes m}$, $N$ will still act trivially so $V^{\otimes m}$ cannot cover $\Irrep(G)$. Moreover $M_G(V)=\frac{1}{|N|}=\Omega(1)$ because irreducible $G/N$ representations are also irreducible $G$-representations. Therefore $V$ contradicts condition \ref{item:tensorpower}. 

For the main part of the proof, take $N\trianglelefteq G$ with $|G/N|=O(1)$ and with non-trivial center $K\subseteq N$. We will construct a $G$-representation $V$ with $M_G(V)=\Omega(1)$ such that $M_G(V^{\otimes m})\leq \frac{1}{2}$. The idea is to construct $G$-representations from $K$-characters via induction. Letting $\theta:K\to\mathbb C^{\times}$ be a multiplicative character, we define $V_{\theta}=\Ind_K^N\theta$. From the Mackey formula and centrality of $K$ in $N$,

\[\chi^{V_{\theta}}(g)=\left\{\begin{array}{lr}
        0, & \text{for } g\notin K\\
        \frac{|N|}{|K|}\cdot \theta(g), & \text{for } g\in K
        \end{array}\right\}\]

It follows that the induced representations $V_{\theta}$ remain mutually orthogonal and satisfy $\bigoplus_{\theta\in K^*}V_{\theta} = \chi^{Reg}_N.$ This implies that they are of the form \[V_{\theta}:=\bigoplus_{\llambda\in \Lambda_{\theta}\subseteq \Irrep(N)} (\dim \llambda)\llambda\] where $\{\Lambda_{\theta}:\theta\in K^*\}$ is a partition of $\Irrep(N)$. Hence their $N$-Plancherel measures satisfy \[M_N(V_{\theta})=\frac{\dim(V_{\theta})}{|N|}=\frac{1}{|K|}.\]

We induct again from $N$ to $G$, obtaining $G$-representations

\[W_{\theta}:=\Ind_K^G \theta=\Ind_N^G V_{\theta}.\]

Observe that any inner automorphism on $G$ restricts to an automorphism on $N$, hence preserves its center $K$. Therefore conjugation defines a group homomorphism $\varphi_0:G\to Aut(K)$. $\varphi_0(n)$ acts trivially for any $n\in N$ so we obtain a quotient map $\varphi:G/N\to Aut(K)$. The Mackey formula again takes a simple form. Writing $\varphi_g(\theta)(\cdot)=\theta(\varphi(g)(\cdot))$ we may view $\varphi$ as a $G/N$ action on $K^*$. Then

\begin{equation}\label{eq:ind2}\chi^{W_{\theta}}(k)=\frac{|N|}{|K|} \sum_{g_i\in G/N} \varphi_{g_i}(\theta)(k).\end{equation}

Partition the set $K^*$ of $K$-characters into equivalence classes $(S_i)_{i=1}^j$ by the action of $G/N$. \eqref{eq:ind2} implies that if $\theta,\theta'$ are in the same equivalence class then $W_{\theta}=W^{\theta'}$, while if not then $W_{\theta},W^{\theta'}$ have orthogonal characters. Let $W^{S_i}=W_{\theta}$ if $\theta\in S_i$. Next observe that $\bigoplus_{\theta\in K^*} \theta\simeq V^{Reg}_K$ and so \[\bigoplus_{i\leq j} |S_i|\cdot W^{S_i}=\bigoplus_{\theta\in K^*}W_{\theta}=\Ind_K^G V^{Reg}_K=V^{Reg}_G.\] Since the representations $|S_i|\cdot W^{S_i}$ have pairwise orthogonal characters and their direct sum is the regular representation, they must be of the form 

\[|S_i|\cdot W^{S_i}=\sum_{\llambda\in \Lambda_i\subseteq \Irrep(G)}(\dim\llambda)\llambda\]

for some partition $\{\Lambda_i:i\in [j]\}$ of $\Irrep(G)$. Therefore $M_G(W^{S_i})=\frac{|S_i|\dim(W^{S_i})}{|G|}=\frac{|S_i|}{|K|}.$ Altogether for any $\varphi$-invariant subset $S\subseteq K^*$, letting $V^S=\bigoplus_{\theta\in S} \theta$ we have 

\[M_G(\Ind_K^G(V^S))=M_K(V^S).\]

By Lemma~\ref{lem:sumset}, there exists $A\subseteq K^*$ invariant under the action $\varphi$ of $G/N$ of size $|A|=\Omega_{k,m}(|K|)$, and which also satisfies $|mA|\leq \frac{|K^*|}{2}$. We take $V=\bigoplus_{\theta\in A}W_{\theta}$. It is not difficult to see by $\varphi$-invariance of $A$ that $\widetilde{\chi}^V(g)$ vanishes outside of $K$, and that for any $k\in K$:

\[\widetilde{\chi}^{V}(k)=\frac{|N|}{|G|}\sum_{\theta\in A} \theta(k).\]

Hence the $m$-th tensor power of this character is in the $\mathbb C$-linear span of $\{\chi^{W_{\theta}}|\theta\in mA\}$. As a result, $M_G(V^{\otimes m})\leq \frac{1}{2}$, completing the proof.
\end{proof}

\section{Application to Tensor Product Markov Chains}\label{sec:markov}

Here we prove Corollary~\ref{cor:mix} on tensor product Markov chains. We consider the chains $\cdot\otimes \widetilde V$ for $V$ a $G$-representation, where the reduced representation $\widetilde V$ is as in Definition~\ref{defn:reduced}. To take a step in these Markov chains, from a starting representation $\llambda$ one samples from the irreducible subrepresentations of $\llambda\otimes \widetilde V$, weighted by the product of their multiplicity and dimension. The Plancherel measure is a stationary distribution for any such Markov chain, and the distribution $p_t(\llambda,\cdot)$ after $t$ steps can be generated by sampling in the same way from the irreducible subrepresentations of $\llambda\otimes (\widetilde V)^{\otimes t}$. Recall also that the $\ell^{\infty}$ (or uniform) $\varepsilon$-mixing time is given by $\inf\left\{t\geq 0:\max_{\llambda,\mmu}\left|\frac{p_t(\llambda,\mmu)}{M_G(\mmu)}-1\right|\leq\varepsilon\right\}$ while the $\ell^1$ (or total variation) $\varepsilon$-mixing time is $\inf\left\{t\geq 0:\max_{\llambda,\mmu}\left|p_t(\llambda,\mmu)-M_G(\mmu)\right|\leq\varepsilon\right\}$. There are several other related definitions for mixing times, such as the $\ell^p$ mixing time and the separation distance mixing time. Among all of these choices, uniform mixing time is the largest while total variation mixing time is the smallest. Therefore in the setting of Corollary~\ref{cor:mix}, when $G$ is TQR all of these chains mix within $3$ steps, while when $G$ is not TQR none of these chains are guaranteed to mix within $O(1)$ steps.

\begin{proof}[Proof of Corollary~\ref{cor:mix}]

We begin with the first statement. Let $\llambda,V$ be $G$-representations with $\llambda$ irreducible and $V$ (without loss of generality) reduced. Then:

\[|\widetilde{\chi}^{V^{\otimes 3}}_0\otimes \widetilde{\chi}^{\llambda_i}_0|_1\leq M_{G}(\llambda)\cdot |\widetilde{\chi}^{V^{\otimes 3}}_0|_1 \leq |\widetilde{\chi}^{V}_0|_2\cdot |\widetilde{\chi}^{V}_0|_2\cdot |\widetilde{\chi}^{V}_0|_{\infty}
\leq M_{G}(\llambda_i)\cdot c(G)^{-1/2}.\]

On the other hand,

\[\widetilde{\chi}^{V^{\otimes 3}}(e)\widetilde{\chi}^{\llambda}(e)=M_{G}(\llambda)\cdot M_{G}(V)^3.\]

Let $p_3(\llambda,\mmu)$ denote the probability to reach $\mmu$ from $\llambda$ in exactly $3$ steps. Then with $\propto$ indicating proportionality as $\mmu$ varies over $\Irrep(G)$,

\[p_3(\llambda,\mmu) \propto \langle \widetilde{\chi}^{\mmu},\widetilde \chi^{V^{\otimes 3}}\otimes \chi^{\llambda}\rangle
\propto M_{G}(\mmu)M_{G}(\llambda)M_{G}(V)^3+\langle \widetilde{\chi}^{\mmu}_0,\widetilde \chi^{V^{\otimes 3}}_0\otimes \chi^{\llambda}_0\rangle.\]

As $M_{G}(V)=\Omega(1)$ and $G$ is TQR, it follows that $c(G)^{-1/2}=o(M_{G}(V)^3).$ Recall:

\[\langle \widetilde{\chi}^{\mmu}_0,\widetilde \chi^{V^{\otimes 3}}_0\otimes \chi^{\llambda}_0\rangle \leq |\widetilde{\chi}^{\mmu}_0|_{\infty}|\widetilde \chi^{V^{\otimes 3}}_0\otimes \chi^{\llambda}_0|_1\leq M_{G}(\mmu) M_{G}(\llambda)\cdot c(G)^{-1/2} .\]

We conclude that for $c(G)$ large,

\[p_3(\llambda,\mmu) \propto M_{G}(\mmu)(1+o(1)).\]

This implies the uniform mixing time result. For total variation non-mixing, we recall Theorem~\ref{thm:TQR}, assertion $\ref{item:tensorpower}$. Taking the contrapositive we see that if $G$ is not TQR, then there is a $G$-representation $V$ with $M_{G}(V)=\Omega_m(1)$ but $M_G(V^{\otimes m})\leq \frac{1}{2}$ for any fixed $m$. Taking as starting point $\llambda$ the trivial representation,  after $m$ steps at least half of $\Irrep(G)$ is still completely inaccessible to the tensor product Markov chain. Such a distribution must have large total variation distance from the Plancherel measure stationary distribution, which concludes the proof.

\end{proof}

\section{Relations between TQR and quasi-random groups}

Here we make some comments on the structure of TQR and quasi-random groups. We observe that any quasi-random group has a quotient which is both quasi-random and TQR, but there seems to be no analogous general way to go from a TQR group to a quasi-random group.

\begin{cor}

Any large center-free quasi-random group $G$ is TQR.

\end{cor}

\begin{proof}

If $G$ is not TQR, then it contains a nontrivial conjugacy class of size $k=O(1)$. As $G$ is center-free, $k>1$. Therefore we obtain a nontrivial conjugacy action $G\to S_k$ with nontrivial normal kernel $N\subsetneq G$. Then $|G/N|\leq k!$, so $G$ contains a constant-sized quotient, contradicting quasi-randomness.

\end{proof}

\begin{cor}
\label{cor:QRtoTQR}
Any large quasi-random group $G$ has a quotient $H$ of super-constant size which is simultaneously quasi-random and TQR.

\end{cor}

\begin{proof}

Begin from $G$ and repeatedly quotient out the center until a center-free quotient $H$ of $G$ is reached. Being quasi-random $G$ contains no non-trivial abelian quotient, so $|H|>1$ as reaching $|H|=1$ requires coming from an abelian quotient in the previous step. $G$ also contains no constant-size non-trivial quotient, implying $|H|=\omega(1)$ is super-constant. Since all quotients of $H$ are quotients of $G$, condition 4 of Theorem~\ref{thm:QR} implies that $H$ is also quasi-random. Moreover $H$ is center-free, hence it is TQR as well.

\end{proof}

\begin{prop}

Let $G=\mathbb F_p\rtimes \mathbb F_p^*$ be the group of affine bijections $x\to ax+b$ on $\mathbb F_p$. Then $G$ is TQR, but none of its subgroups or quotients are quasi-random.

\end{prop}

\begin{proof}

The TQR property follows by considering conjugacy classes so we focus on the latter assertions. First we show that any non-trivial subgroup $H$ of $G$ has a non-trivial abelian quotient, implying $G$ has no quasi-random subgroup. This holds because restriction of the quotient map $G\to\mathbb F_p^*$ defined by $(ax+b)\to a$ gives an abelian quotient unless $H$ is contained in the set $\{x+b\}$. In the latter case $H$ is abelian already. 

Next we show that $G$ cannot have a quasi-random quotient, and in fact that all non-trivial quotients of $G$ are abelian. Indeed any irreducible representation of a quotient $G/N$ pulls back to an irreducible representation of $G$. As $G$ has order $p^2-p$ and has irreducible representations of dimension only $p-1$ and $1$, simple size considerations imply only the $1$-dimensional irreducible $G$-representations can be irreducible representations of a non-trivial quotient. Therefore any quotient of $G$ is abelian, concluding the proof.

\end{proof}

\small

\bibliographystyle{alpha}

\bibliography{main}

\appendix

\section{Statement of Theorem~\ref{thm:QR}}

Here we justify our statement of Theorem~\ref{thm:QR}. In \cite[Theorems 3.3, 4.5, 4.6, 4.8]{qrgroups} the following statements are shown to be equivalent for a finite group $G$:

\begin{enumerate}[label=(\Alph*)]
    \item $G$ has no $O(1)$-dimensional non-trivial irreducible representations.\label{it:A}
    \item If $A_1,A_2,A_3\subseteq G$ each have size $\Omega(|G|)$, then $a_1a_2=a_3$ has a solution in $(a_1,a_2,a_3)\in A_1\times A_2\times A_3.$\label{it:B}
    \item $G$ has neither an $O(1)$-size nor an abelian non-trivial quotient.\label{it:C}
\end{enumerate}

It is easy to see that Assertion~\ref{it:B} is equivalent to assertion $2$ of Theorem~\ref{thm:QR}. As Assertion $2$ of Theorem~\ref{thm:QR} clearly implies Assertion 3, it only remains to show that Assertion $3$ of Theorem~\ref{thm:QR} implies one of the others. We will show it implies Assertion $1$. Going by contrapositive, we suppose we are given a nontrivial homomorphism $\varphi:G\to U(k)$ for $k=O(1)$. Take the set $A$ to be the preimage by $\varphi$ of a small neighborhood of the identity in $U(k)$. Then $A^m$ contains no element with negative trace, but $\varphi(G)$ must contain such elements as the average trace of $\varphi(g)$ for $g\in G$ uniformly random is $0$. Moreover a simple volume argument on cosets of $A$ shows that $\frac{|A|}{|G|}=\Omega_{k,m}(1)$. This proves Assertion $1$ of Theorem~\ref{thm:QR} from Assertion $3$.

\end{document}